\documentclass[11pt]{amsart}

\usepackage{amsfonts, amsmath, amscd}
\usepackage[psamsfonts]{amssymb}

\usepackage{amssymb}

\usepackage{pb-diagram}

\usepackage[all,cmtip]{xy}

\usepackage[usenames]{color}

\newtheorem{theorem}{Theorem}[section]
\newtheorem{lemma}[theorem]{Lemma}
\newtheorem{corollary}[theorem]{Corollary}

\newtheorem{remark}[theorem]{Remark}
\newtheorem{proposition}[theorem]{Proposition}
\newtheorem{definition}[theorem]{Definition}

\newcommand{\CC}{C_k}

\newcommand{\GG}{\mathfrak{G}}
\newcommand{\dd}{(\mathbf{D})}

\newcommand{\w}{\omega}

\newcommand{\FF}{\mathcal{F}}
\newcommand{\KK}{\mathcal{K}}

\newcommand{\Nn}{\mathcal{N}}

\newcommand{\Pp}{\mathfrak{P}}

\renewcommand{\phi}{\varphi}

\newcommand{\U}{\mathcal U}

\newcommand{\cl}{\mathrm{cl}}

\newcommand{\gglim}{\mbox{g-}\underrightarrow{\lim}\,}
\newcommand{\glim}{\mbox{{\em g}-}\underrightarrow{\lim}\,}

\title[Topological properties of inductive limits]{Topological properties of inductive limits of \\ closed towers of mertrizable groups}

\author[S.~Gabriyelyan]{Saak Gabriyelyan}
\address{Department of Mathematics, Ben-Gurion University of the Negev, Beer-Sheva, P.O. 653, Israel}

\email{saak@math.bgu.ac.il}



\begin{document}

\begin{abstract}
Let  $\{ G_n\}_{n\in\w}$ be a closed tower of metrizable groups.
Under a mild condition called  $(GC)$  and which is strictly weaker than $PTA$ condition introduced in \cite{TSH}, we show that: (1) the inductive limit $G=\gglim G_n$ of the tower is a Hausdorff group, (2) every $G_n$ is a closed subgroup of $G$, (3) if $K$ is a compact subset of $G$, then $K\subseteq G_m$ for some $m\in\w$, (4)  $G$ has a $\GG$-base and  countable tightness, (5) $G$ is an $\aleph$-space, (6) $G$ is an Ascoli space if and only if either (i) there is  $m\in\w$ such that $G_n$ is open in $G_{n+1}$ for every $n\geq m$, so $G$ is metrizable; or (ii) all the groups $G_n$ are locally compact and $G$ is a sequential non-Fr\'{e}chet--Urysohn space.
\end{abstract}

\maketitle




\section{Introduction}


Let $\{ (G_n,\tau_n)\}_{n\in\w}$ be a tower of topological groups, i.e., $G_n$ is a subgroup of $G_{n+1}$ for every $n\in\w:=\{ 0,1,2,\dots\}$.  The {\em inductive limit $\glim G_n$} of the tower is the union $G=\bigcup_{n\in\w} G_n$ endowed with the finest (not necessarily Hausdorff) group topology $\tau_{gr}$ such that all the identity inclusions $G_n\to G$ are continuous. Besides the topology $\tau_{gr}$, the union $G$ carries the topology $\tau_{ind}$ of the inductive limit of  $\{ G_n\}_{n\in\w}$ in the category of topological spaces and continuous mappings.  Tatsuuma, Shimomura and Hirai showed in \cite{TSH} by two counterexamples that $\tau_{ind}$ is not necessarily a group topology for $G$. However, for the important case when all $G_n$ are locally compact, they proved that  $\tau_{ind}$ is indeed a group topology. If additionally all the groups $G_n$ are  metrizable and the  tower $\{ G_n\}_{n\in\w}$ is  {\em closed} (i.e., $G_n$ is a closed subgroup of $G_{n+1}$ for all $n\in\w$), Yamasaki \cite{Yamasaki} proved the converse assertion by showing that if $\tau_{ind}$ is a group topology, then  all the groups $G_n$ are locally compact (we reprove this result in Theorem \ref{t:Ascoli-inductive-1} below).
This pathological phenomena rises the problem of explicit description of the topological structure of the inductive limit $\gglim G_n$. This problem was discussed in \cite{Edamatsu,HSTH,TSH}.

The most natural construction of a topology on $G$ which takes into account the group topology $\tau_{gr}$ on $G$ is the so called {\em Bamboo-Shoot topology} $\tau_{BS}$ defined in \cite{TSH}.
For every $n\in\w$, denote by $\Nn^s_n$ the family of all $\tau_n$-open symmetric neighborhoods of the identity $e\in G_n$ and set $\mathfrak{N} := \left\{ (U_n)_{n\in\w} \in \prod_{n\in\w} \Nn^s_n \right\}$.
For every sequence $(U_n)_{n\in\w} \in \mathfrak{N}$ and each $k\in\w$ such that  $k\leq n$, set
\[
\begin{aligned}
U^+(n,k) & := U_k\cdot U_{k+1}\cdots U_n,\\
U(n,k) & := \big( U^+(n,k)\big)^{-1} \cdot U^+(n,k)= U_n\cdots U_{k+1}\cdot U_k^2\cdot U_{k+1}\cdots U_n ,\\
\mathrm{U}^+[k] & := \bigcup_{n=k}^\infty U^+(n,k) = \bigcup_{n=k}^\infty U_k\cdot U_{k+1}\cdots U_n,\\
\mathrm{U}[k] & := \big( \mathrm{U}^+[k] \big)^{-1} \cdot \mathrm{U}^+[k] = \bigcup_{n=k}^\infty U_n\cdots U_{k+1}\cdot U_k^2\cdot U_{k+1}\cdots U_n,
\end{aligned}
\]
the set $\mathrm{U}[k]$ is called a {\em Bamboo-Shoot} ($BS$ for short) neighborhood of $e\in G$. Then the families
\[
\U_l :=\{ g\mathrm{U}[k]: \mathrm{U}[k]\in \Nn_{BS}^s\}\; \mbox{ and }\; \U_r :=\{ \mathrm{U}[k]g : \mathrm{U}[k]\in \Nn_{BS}^s\},
\]
where $\Nn_{BS}^s := \left\{ \mathrm{U}[k]: (U_n)_{n\in\w} \in \mathfrak{N}, k\in\w \right\}$, are open basis of the same topology $\tau_{BS}$ on $G$ called the {\em Bamboo-Shoot topology}. The group $G$ with $\tau_{BS}$ is a  (maybe non-Hausdorff) quasitopological group (see \cite{TSH} or \cite{BR}).
It is easy to see that $\tau_{gr}\leq \tau_{BS} \leq \tau_{ind}$, so the following  natural question is of interest: Under which conditions the Bamboo-Shoot topology $\tau_{BS}$ is a {\em group} topology (in this case clearly $\tau_{gr}= \tau_{BS}$)?  Tatsuuma, Shimomura and Hirai \cite{TSH}  showed that if the tower $\{ G_n\}_{n\in\w}$  satisfies the ``$PTA$ condition'' (see Definition \ref{def:PTA-condition} below), then $\tau_{BS}$ is a group topology on $G$ and hence $(G,\tau_{BS})=\gglim G_n$. Banakh and Repov\v{s} \cite{BR} showed that $\tau_{BS}$ is a group topology if the tower satisfies the ``balanced condition'' (see Definition \ref{def:balanced-condition}). In Section \ref{sec:GC-condition} we study towers of topological groups which satisfy the following condition.

\begin{definition} \label{def:GC-condition} {\em
A tower $\{ G_n\}_{n\in\w}$ of topological groups satisfies the {\em Group Condition} ($(GC)$ for short) if for every $(U_n)_{n\in\w} \in \mathfrak{N}$ and $k\in\w$ there is $(V_n)_{n\in\w} \in \mathfrak{N}$ such that $
\mathrm{V}[k] \subseteq \mathrm{U}^+[k]$.}
\end{definition}

It is easy to show (Proposition \ref{p:BS-GC}) that if a tower  $\{ G_n\}_{n\in\w}$ of topological groups satisfies $(GC)$, then  $\tau_{BS}$ is indeed a group topology. Moreover,  results of \cite{BR} and \cite{TSH} show that the condition $(GC)$ is strictly weaker than the $PTA$ condition and the balanced condition, see Remark \ref{rem:PTA-balanced-GC}. In Theorem \ref{t:IL-properties-compact}, which is the main result of Section \ref{sec:GC-condition},  we show that if a closed  tower $\{ G_n\}_{n\in\w}$  of Hausdorff topological groups satisfies $(GC)$, then the Bamboo-Shoot topology $\tau_{BS}$ is a Hausdorff group topology on $G=\bigcup_{n\in\w} G_n$ such that $G_n$ is a closed subgroup of $(G,\tau_{BS})$. Moreover, if $K$ is a compact subset of $(G,\tau_{BS})$, then $K\subseteq G_m$ for some $m\in\w$. These results are essential for the study of topological properties of inductive limits of towers of {\em metrizable} groups.

Let a closed tower $\{ G_n\}_{n\in\w}$  of metrizable groups satisfy $(GC)$.  If all the groups $G_n=E_n$ are Fr\'{e}chet (=complete metrizable locally convex) spaces, then the Bamboo-Shoot topology $\tau_{BS}$ coincides also with the inductive limit topology in the category of locally convex spaces and continuous linear mappings (see Proposition 3.1 of \cite{HSTH}). In this case the inductive limit $E=\gglim E_n$ is called a {\em strict $(LF)$-space}.
In  \cite{CKS-2002} it is proved that every strict $(LF)$-space $E$ has countable tightness. Moreover, we proved in \cite{GKL2} that $E$ has even the Pytkeev property. Recall that a topological space $X$ has the {\em Pytkeev property} if for each $A\subseteq X$ and  each $x\in \overline{A}\setminus A$, there are infinite subsets $A_1, A_2, \dots $ of $A$ such that each neighborhood of $x$ contains some $A_n$. Every sequential space is Pytkeev.
In \cite{kaksax}, K\c{a}kol and Saxon showed that a strict $(LF)$-space $E$ is a $k$-space if and only if $E$ is sequential if and only if $E$ is Fr\'{e}chet or is $\phi$ (where $\phi$ is the inductive limit of finite dimensional vector spaces, so as we mentioned above, $\tau_{BS}=\tau_{ind}$). This result was essentially strengthened in \cite{Gab-LF} by showing that a strict $(LF)$-space $E$ is an Ascoli space if and only if $E$ is Fr\'{e}chet or is $\phi$. Let us recall (see \cite{BG}) that a Tychonoff space $X$ is an {\em Ascoli space} if and only if every compact subset of $\CC(X)$ is equicontinuous, where $\CC(X)$ is the space $C(X)$ of all real-valued functions on $X$ endowed with the compact-open topology. By the classical Ascoli theorem, every $k$-space is an Ascoli space.

Besides the sequential and compact type properties mentioned in the previous paragraph, there are important topological properties coming for example from the generalized metric space theory. We recall only two the most important properties. Let $\Nn$ be a family  of subsets of a topological space $X$. Following \cite{Arhan}, the family $\Nn$ is a {\em network at a point} $x\in X$ if for each neighborhood $O_x$ of $x$ there is a set $N\in\mathcal{N}$ such that $x\in N\subseteq O_x$; $\Nn$ is a {\em network} in $X$ if $\mathcal{N}$ is a network at each point $x\in X$. Following Michael \cite{Mich}, $\Nn$ is a {\em $k$-network} in $X$ if whenever $K\subseteq U$ with $K$ compact and $U$ open in $X$, then $K\subseteq \bigcup\FF\subseteq U$ for some finite $\FF\subseteq\Nn$.
Okuyama \cite{Oku} and O'Meara \cite{OMe2}, having in mind the Nagata--Smirnov metrization theorem, introduced the classes of $\sigma$-spaces and $\aleph$-spaces. A topological space $X$ is called a {\em $\sigma$-space} (respectively, an {\em $\aleph$-space}) if $X$ is regular and has a $\sigma$-locally finite (respectively, $k$-)network.

Being motivated by the study of $(DF)$-spaces, $C(X)$-spaces and spaces in  the class $\GG$ in the sense of Cascales and Orihuela, the concept of a $\GG$-base has been formally introduced in \cite{FKLS} in the realm of locally convex spaces. In a more general situation of topological groups this concept was introduced and thoroughly studied in \cite{GKL}. A topological group $G$ has a {\em $\GG$-base} if it has a base $\{ U_\alpha : \alpha\in\w^\w\}$ of neighborhoods at the identity $e\in G$ such that $U_\beta \subseteq U_\alpha$ whenever $\alpha\leq\beta$ for all $\alpha,\beta\in\w^\w$, where $\alpha=(\alpha(n))_{n\in\w}\leq \beta=(\beta(n))_{n\in\w}$ if $\alpha(n)\leq\beta(n)$ for all $n\in\w$. It is easy to see that every metrizable group has a $\GG$-base. It is also known that every  strict $(LF)$-space has a $\GG$-base, see \cite{CKS}.

The following diagram describes the relation between the aforementioned properties:
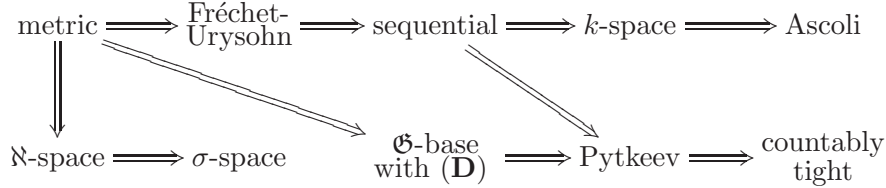
\[
\xymatrix{
\mbox{metric} \ar@{=>}[r]\ar@{=>}[d] \ar@{=>}[rrd]  &  {\substack{\mbox{Fr\'{e}chet-} \\ \mbox{Urysohn}}} \ar@{=>}[r] &  \mbox{sequential} \ar@{=>}[r]\ar@{=>}[rd] &  \mbox{$k$-space}  \ar@{=>}[r] & \mbox{Ascoli} \\
\mbox{$\aleph$-space}  \ar@{=>}[r] & \mbox{$\sigma$-space}  & {\substack{\mbox{$\GG$-base} \\ \mbox{with $\dd$ }}} \ar@{=>}[r] & {\mbox{Pytkeev}} \ar@{=>}[r] & {\substack{\mbox{countably} \\ \mbox{tight}}}
}
\]
(the condition $\dd$ will be defined in the beginning of Section \ref{sec:main}). None of these implications is reversible, see \cite{BG,Eng,GK-GMS1,Mich}.

The above  mentioned  results  motivate the study of topological properties of inductive limits of closed tower $\{ G_n\}_{n\in\w}$  of metrizable groups. This is the main theme of the paper. In Section \ref{sec:main}, under the assumption that  $\{ G_n\}_{n\in\w}$ satisfies $(GC)$, we show that the inductive limit $G=\gglim G_n$ has the following properties: (1)  $G$ has a $\GG$-base which satisfies $\dd$ and is an $\aleph$-space (Theorem \ref{t:Ascoli-inductive-1}), and (2) $G$ is an Ascoli space if and only if either (i) there is  $m\in\w$ such that $G_n$ is open in $G_{n+1}$ for every $n\geq m$, so $G$ is metrizable; or (ii) all the groups $G_n$ are locally compact and $G$ is a sequential non-Fr\'{e}chet--Urysohn space (Theorem \ref{t:IS-Pytkeev-G-base}).

\section{Topological properties of inductive limits of towers of topological groups} \label{sec:GC-condition}


Recall that a group $G$ endowed with a topology $\tau$ is a {\em semitopological group} if the multiplication $m:G\times G\to G, (g,h)\mapsto gh$, is separately continuous; $(G,\tau)$ is a {\em quasitopological group} if it is a semitopological group with continuous inversion $(\cdot)^{-1}: G\to G, g\mapsto g^{-1}$.

\begin{lemma} \label{l:BS-description}
Let $\{ G_n\}_{n\in\w}$ be a tower of topological groups. Then:
\begin{enumerate}
\item[{\rm (i)}] {\em (\cite{TSH})} $\tau_{BS} \leq \tau_{ind}$.
\item[{\rm (ii)}] {\em (\cite{BR})} The identity map $(G,\tau_{BS})\to \glim G_n$ is continuous. Consequently, if $\tau_{BS}$ is a group topology on $G$, then $(G,\tau_{BS})= \glim G_n$.
\item[{\rm (iii)}] {\em (\cite{TSH})}  $(G,\tau_{ind})$ and  $(G,\tau_{BS})$ are quasitopological groups.
\end{enumerate}
\end{lemma}

Our interest in the condition $(GC)$ is explained by the following assertion.

\begin{proposition} \label{p:BS-GC}
If a tower $\{ G_n\}_{n\in\w}$ of topological groups satisfies $(GC)$, then $\tau_{BS}$ is a  (maybe non-Hausdorff) group topology on $G:=\bigcup_{n\in\w} G_n$. Thus $(G,\tau_{BS})= \glim (G_n,\tau_n)$.
\end{proposition}

\begin{proof}
Since $(G,\tau_{BS})$ and $\gglim (G_n,\tau_n)$ are semitopological groups (Lemma \ref{l:BS-description}),
it suffices to show that the multiplication $m:G\times G\to G$ is $\tau_{BS}$-continuous at the identity $e\in G$. Let $\mathrm{U}[0]\in \Nn_{BS}^s$ be defined by $(U_n)_{n\in\w} \in \mathfrak{N}$. By $(GC)$, choose  $\mathrm{V}[0]\in \Nn_{BS}^s$ such that $\mathrm{V}[0]\subseteq \mathrm{U}^+[0]$, and hence $\mathrm{V}[0]\subseteq (\mathrm{U}^+[0])^{-1}$. Then
\[
\mathrm{V}[0]\cdot \mathrm{V}[0] \subseteq (\mathrm{U}^+[0])^{-1} \cdot\mathrm{U}^+[0] =\mathrm{U}[0].
\]
Thus $m$ is continuous and $\tau_{BS}$ is a group topology on $G$.
\end{proof}
We do not know whether the condition $(GC)$ is also necessary for $\tau_{BS}$ of being a group topology.

Now we compare the condition $(GC)$ with the condition $PTA$ and the balanced condition. Let us recall their definitions.

\begin{definition}[\cite{TSH}] \label{def:PTA-condition}   {\em
A tower $\{ G_n\}_{n\in\w}$ of topological groups  is said to satisfy {\em $PTA$} (={\em Passing Through Assumption}) if, for every $n\in\w$, the group $G_n$ has a neighborhood base $\mathcal{B}_n$ at the identity $e$, consisting of open symmetric neighborhoods $U\subseteq G_n$ such that for every $m\geq n$ and every neighborhood $W\subseteq G_m$ of $e$ there exists a neighborhood $V \subseteq G_m$ of $e$ such that $VU \subseteq UW$.}
\end{definition}

It is proved in Lemma 2.3 of \cite{TSH} that if a tower $\{ G_n\}_{n\in\w}$ of topological groups  satisfies $PTA$, then $\tau_{BS}$ is a group topology weaker than $\tau_{ind}$. For example, if the tower  consists of balanced groups, then  it satisfies $PTA$.

The balanced condition was defined in \cite{BR} being motivated by towers of balanced groups. Recall that a topological group $G$ is called {\em balanced} if it has a neighborhood base $\{ U_i\}_{i\in I}$ at the identity $e\in G$ consisting of $G$-invariant neighborhoods (i.e., $gU_i g^{-1}=U_i$ for every $g\in G$ and $i\in I$). A triple $(H,\Gamma, G)$ of topological groups $H\leq \Gamma\leq G$ is called {\em balanced} if for every neighborhoods $V\subseteq \Gamma$ and $U\subseteq G$ of the identity $e\in  G$ the product
\[
V\cdot \sqrt{U}^H, \mbox{ where } \sqrt{U}^H:=\{ g\in G: hgh^{-1}\in U \mbox{ for every } h\in H\},
\]
is a neighborhood of $e$.

\begin{definition}[\cite{BR}]  \label{def:balanced-condition}   {\em
A tower $\{ G_n\}_{n\in\w}$ of topological groups  is called {\em balanced} if each triple $(G_n,G_{n+1}, G_{n+2})$, $n\in\w$, is balanced.}
\end{definition}

\begin{proposition} \label{p:PTA-balanced-GC}
If a tower $\{ G_n\}_{n\in\w}$ of topological groups satisfies $PTA$ or is balanced, then it satisfies $(GC)$.
\end{proposition}

\begin{proof}
If $\{ G_n\}$ satisfies $PTA$, then it satisfies $(GC)$ by Lemma 3.3 of \cite{TSH}. If $\{ G_n\}$ is balanced, then it satisfies $(GC)$ by Theorem 4.2 of \cite{BR}.
\end{proof}

\begin{remark} \label{rem:PTA-balanced-GC} {\em
It is shown in Theorem 6.1 and Example 6.2 of \cite{BR}, that there are balanced towers which do not satisfy $PTA$ and non-balanced towers which satisfy $PTA$. Therefore, by Proposition \ref{p:PTA-balanced-GC}, the condition $(GC)$ is strictly weaker than the condition of being $PTA$ and the condition of being balanced. \qed}
\end{remark}


To prove the main result of this section we need the following lemma.
\begin{lemma} \label{l:IL-1}
Let $H$ be a closed subgroup of a topological group $G$ and let $a,b\in G$ be such that $a\not\in H$ and $b\in H$. Then there is a neighborhood $W$ of the unit $e\in G$ such that $a\not\in b\cdot W HW$.
\end{lemma}

\begin{proof}
Define a map $p:G\times G\to G$ by $p(t,g):= t^{-1} b^{-1} a g^{-1}$. Then $p$ is continuous and $p(e,e)=b^{-1} a \not\in H$. Since $H$ is closed, there exists a neighborhood $V$ of $p(e,e)$ such that $V\cap H=\emptyset$. Choose a neighborhood $W$ of $e$ such that $p(W,W) \subseteq V$. Then, for every $t,g\in W$, we obtain $t^{-1} b^{-1} a g^{-1}\not\in H$ or $a\not\in btHg$. This means that $a\not\in b WHW$.
\end{proof}

\begin{theorem} \label{t:IL-properties-compact}
Let a closed  tower $\{ (G_n,\tau_n)\}_{n\in\w}$  of  Hausdorff  topological groups satisfy $(GC)$. Set $G:=\bigcup_{n\in\w} G_n$. Then:
\begin{enumerate}
\item[{\rm (i)}] $\tau_k = \tau_{BS}|_{G_k}$ and $(G_k,\tau_k)$ is $\tau_{BS}$-closed;
\item[{\rm (ii)}] $\tau_{BS}$ is a Hausdorff group topology on $G$, so $(G,\tau_{BS})= \glim (G_n,\tau_n)$;
\item[{\rm (iii)}] if $K$ is a compact subset of $(G,\tau_{BS})$, then $K\subseteq G_m$ for some $m\in\w$.
\end{enumerate}
\end{theorem}

\begin{proof}
(i) Our proof follows the proof of Proposition 3.2 of \cite{TSH}. Let $D$ be a $\tau_k$-closed subset of $G_k$ (in particular, $D=G_k$). We have to show that $D$ is also $\tau_{BS}$-closed. Fix arbitrarily an element $g\in G\setminus D$. Choose an $n>k$ such that $g\in G_n$. Then, for every $j\geq n$, $g^{-1} D$ is closed in $G_j$ and $e\not\in g^{-1}D$. For $j=n$, choose  $U_n \in \Nn_n^s$ such that $U_n^2 \cap g^{-1} D=\emptyset$. By induction on $j=n,\dots, m$, for $j=m+1$ choose $U_{m+1} \in \Nn_{m+1}^s$ such that $U_{m+1}^2 \cap G_m \subseteq U_m$. Then (recall that $U_n \cdots U_m \subseteq G_m$)
\[
\begin{aligned}
 (U_n \cdot U_{n+1} \cdots U_m \cdot U_{m+1}^2) \cap G_n
& \subseteq \big( (U_n \cdot U_{n+1} \cdots  U_m) \cdot  (U_{m+1}^2 \cap G_m)\big) \cap G_n \\
&  \subseteq ( U_n \cdot U_{n+1} \cdots  U_m^2) \cap G_n \subseteq \cdots \subseteq U_n^2 \cap G_n =U_n^2.
\end{aligned}
\]
Define $U:=\mathrm{U}^+[n]= \bigcup_{m\geq n} U_n\cdots U_m$. Then, by $(GC)$, $U$ is a $\tau_{BS}$-neighborhood of $e\in G$, and, the above inclusions imply (recall that $n>k$ and $D\cup\{ g\}\subseteq G_n$)
\[
gU\cap D = (gU\cap G_n) \cap D =g(U\cap G_n) \cap D \subseteq gU_n^2 \cap D=\emptyset.
\]
Therefore $D$ is closed in $\tau_{BS}$. Thus $\tau_k \subseteq \tau_{BS}|_{G_k}$, and hence $\tau_k = \tau_{BS}|_{G_k}$.

(ii) By the proof of (i), every one-point subset of $G$ is $\tau_{BS}$-closed. Therefore $\tau_{BS}$ is a $T_0$-topology. Now Theorem 4.8 of \cite{HR1} and Proposition  \ref{p:BS-GC} imply that $\tau_{BS}$ is a Hausdorff group topology.

(iii) Suppose for a contradiction that,  for every $n\in\w$, the compact set $K$ is not contained in $G_n$. Then we can find a sequence $0<n_0 <n_1<\cdots$ in $\w$ and a sequence $\{ a_k\}_{k\in\w} $ in $K$ such that $a_k\in G_{n_k}\setminus G_{n_k-1}$ for every $k\in\w$.

By (i) and (ii) and Lemma \ref{l:IL-1}, for every $k\in\w$, choose $W_k\in \Nn^s_G$ such that
\begin{equation} \label{equ:IL-1}
W_{k+1}^2 \subseteq W_k \; \mbox{ and } \; a_k \not\in \bigcup_{i=0}^{k -1} a_i \cdot W_k^2 G_{n_k -1} W_k^2.
\end{equation}
For every $k,i\in\w$, choose $W_{i,k}\in \Nn^s_G$ such that
\begin{equation} \label{equ:IL-2}
W_{0,k}^2 \subseteq W_k \; \mbox{ and } \; W_{i+1,k}^2 \subseteq W_{i,k}.
\end{equation}
For every $i=0,\dots,n_0-1$, set $U_i :=G_i$. If $j\in \w$ and $i=n_j,\dots,n_{j+1}-1$, choose $U_i\in \Nn_{i}^s$ such that
\begin{equation} \label{equ:IL-3}
U_i \subseteq  W_{i,j}.
\end{equation}
Set $\mathrm{U}[0]:= \bigcup_{n\in\w} U(n,0)$.

Fix an $l\in\w$. For every $n>n_l$ choose $s\in\w$ such that $n_s \leq n<n_{s+1}$, so $l\leq s$. Then
\[
\begin{aligned}
\prod_{i=n_j}^{n_{j+1}-1} U_{n_{j+1}-1 +n_j -i} & = U_{n_{j+1}-1} \cdot U_{n_{j+1}-2} \cdots U_{n_{j}} \\
& \stackrel{(\ref{equ:IL-3})}{\subseteq}  W^2_{n_{j+1}-1,j} \cdot W_{n_{j+1}-2,j} \cdots W_{n_{j},j} \stackrel{(\ref{equ:IL-2})}{\subseteq}  W^2_{n_{j},j} \subseteq  W_j
\end{aligned}
\]
and
\[
\begin{aligned}
\prod_{i=n_j}^{n_{j+1}-1} U_i & = U_{n_{j}}\cdot U_{n_{j}+1} \cdots U_{n_{j+1}-1} \\
& \stackrel{(\ref{equ:IL-3})}{\subseteq}  W_{n_j,j} \cdot W_{n_j+1,j} \cdots W_{n_{j+1} -1,j}^2 \stackrel{(\ref{equ:IL-2})}{\subseteq}   W^2_{n_{j},j} \subseteq  W_j,
\end{aligned}
\]
and hence
\[
\begin{aligned}
U(n,0) & \subseteq \prod_{j=l}^s \left(\prod_{i=n_j}^{n_{j+1}-1} U_{n_{j+1}-1 +n_j -i} \right) \cdot \prod_{i=0}^{n_l-1} U_{n_l-1 -i} \cdot \prod_{i=0}^{n_l-1} U_i \cdot \prod_{j=l}^s \left(\prod_{i=n_j}^{n_{j+1}-1} U_i \right) \\
& \subseteq \prod_{j=l}^s W_{s+l-j} \cdot G_{n_l -1} \cdot \prod_{j=l}^s W_j \stackrel{(\ref{equ:IL-1})}{\subseteq} W_l^2 \cdot G_{n_l -1} \cdot W_l^2.
\end{aligned}
\]
Therefore, for every $l\in\w$, we obtain
\begin{equation} \label{equ:IL-4}
\mathrm{U}[0]= \bigcup_{n>n_l} U(n,0) \subseteq W_l^2 \cdot G_{n_l -1} \cdot W_l^2.
\end{equation}
Now, for every $k,l\in\w$  such that $l>k$, (\ref{equ:IL-1}) and (\ref{equ:IL-4}) imply
$
a_l \cdot a_k^{-1} \not\in \mathrm{U}[0].
$
Since $\mathrm{U}[0]$ is a neighborhood of $e$ for the group topology $\tau_{BS}$ we obtain that $K$ is not compact. 
This contradiction finishes the proof. 
\end{proof}

It is well known that the strict inductive limit of complete locally convex spaces is complete, see \cite{bonet}. For the abelian case we prove an analogous result with a similar proof.
\begin{proposition} \label{p:IL-completness}
Let $\{ G_n\}_{n\in\w}$ be a closed  tower of abelian topological groups. Then the group $\glim G_n$ is complete if and only if all the groups $G_n$ are complete.
\end{proposition}

\begin{proof}
We write the addition in $G$ multiplicatively.
First we note that, by Theorem \ref{t:IL-properties-compact}, $G:=\gglim G_n$ is Hausdorff and $G_n$ is a closed subgroup $G$ for every $n\in\w$.
If $G$ is complete, then for every $n\in\w$, the group $G_n$ is complete as a closed subgroup of a complete group. Assume now that all the groups $G_n$ are complete. We have to show that each Cauchy filterbase  $\mathfrak{F}=\{ F_i:i\in I\}$  on $G$ converges.
Fix a filterbase $\U=\{ U_\alpha: \alpha\in A\}$ of open symmetric neighborhoods at the identity $e\in G$. Then the family $\mathfrak{B}=\{ F_i U_\alpha :i\in I, \alpha\in A\}$ is also a Cauchy filterbase on $G$.


We claim that $\mathfrak{B}\cap G_n$ is a Cauchy filterbase on $G_n$ for some $n\in\w$. Indeed, otherwise, for every $n\in\w$ there are $F_n\in \mathfrak{F}$ and $U_n\in \U$ such that $F_n U_n \cap G_n =\emptyset$. For every $n\in\w$, choose $V_n\in \U$ such that
\begin{equation} \label{equ:IL-complete}
V_0^4 \subseteq U_0 \; \mbox{ and } \; V_{n+1}^4 \subseteq V_n \cap U_{n}
\end{equation}
and set $W_n := V_n \cap G_n$. Then $\mathrm{W}[0]$ is an open symmetric neighborhood at $e$ in $G$. Since $\mathfrak{B}$ is a Cauchy filterbase, there is $FU\in \mathfrak{B}$ such that $FU\cdot (FU)^{-1} \cup (FU)^{-1}\cdot FU \subseteq V$. 
 Since $\mathfrak{F}$ is a Cauchy filterbase, some $F\in \mathfrak{F}$ must satisfy $FF^{-1} \cup F^{-1}F \subseteq \mathrm{W}[0]$.

Fix arbitrarily $g\in F$ and choose $n\in\w$ such that $g\in G_n$. For every $h\in F\cap F_n$, we have
\[
g=h\cdot (h^{-1}g) \in F_n \cdot (F^{-1}F)\subseteq F_n \mathrm{W}[0],
\]
and hence $g=f_n \cdot t$ for some $f_n\in F_n$ and $t\in \mathrm{W}[0]$. By the definition of $\mathrm{W}[0]$, there are  $m>n$ and $v_i,u_i\in W_i$ ($i=0,\dots,m)$ such that
\[
t= v_m \cdots v_1 \cdot v_0u_0 \cdot u_1 \cdots u_m.
\]
Thus (we use the commutativity of $G$)
\[
\begin{aligned}
G_n \ni g (v_n \cdots v_1 \cdot v_0u_0 & \cdot u_1 \cdots u_n)^{-1} = f_n \cdot (v_m \cdots v_{n+1})\cdot (u_{n+1} \cdots u_m) \\
& \in f_n \cdot W_{n+1}^2 \cdots W_m^2 \stackrel{(\ref{equ:IL-complete})}{\subseteq} f_n \cdot V_{n+1}^4 \subseteq F_n \cdot U_n.
\end{aligned}
\]
But since  $F_n U_n \cap G_n =\emptyset$ we obtain a desired contradiction. The claim is proved.

Now since $G_n$ is complete, the claim implies that $\mathfrak{B}\cap G_n$ has a limit point $x\in G_n$. Therefore $\mathfrak{B}$ converges to $x$ in $G$, and hence also $\mathfrak{F}$ converges to $x$. Thus $G$ is complete. 
\end{proof}


\section{Main results} \label{sec:main}


Let $G$ be a topological group with a  $\GG$-base  $\U=\{ U_\alpha : \alpha\in\w^\w\}$.  For every $k\in\w$ and $\alpha\in \w^\w$, set
\[
D_k(\alpha) := \bigcap_{\beta\in  I_k(\alpha)} U_\beta, \mbox{ where } I_k(\alpha) := \left\{ \beta\in\w^\w : \beta_i = \alpha_i \mbox{ for } i=0,\dots,k\right\}.
\]
Following \cite{GKL2}, we say that  $\U$ satisfies the {\em condition $\dd$} if $\bigcup_{k\in\w} D_k(\alpha)$ is a neighborhood of $e$ for every $\alpha\in\w^\w$. The condition $\dd$  guaranties that the group $G$ has the Pytkeev property, see Theorem 6 of \cite{GKL2}. 


\begin{theorem} \label{t:IS-Pytkeev-G-base}
Let  a closed tower $\{ (G_n,\tau_n)\}_{n\in\w}$ of metrizable groups satisfy $(GC)$.  Then:
\begin{enumerate}
\item[{\rm (i)}] $\glim G_n$ has a $\GG$-base satisfying $\dd$;
\item[{\rm (ii)}]  $\glim G_n$ is an $\aleph$-space.
\end{enumerate}
\end{theorem}

\begin{proof}
Theorem \ref{t:IL-properties-compact} implies that $\tau_{BS}$ is a group topology. Set $G:=\gglim G_n$.

\smallskip
(i) For every $n\in\w$, let $\{ U_{i,n}: i\in\w\}$ be a decreasing base of open symmetric neighborhoods at the identity $e\in G_n$. For every $\alpha=(\alpha(n))_{n\in\w} \in \w^\w$, set
\begin{equation} \label{equ:IL-5}
W_\alpha := \bigcup_{n\in\w} U_{\alpha(n),n}\cdots U_{\alpha(1),1} \cdot U_{\alpha(0),0}^2 \cdot U_{\alpha(1),1}\cdots U_{\alpha(n),n}.
\end{equation}
Then, by Theorem \ref{t:IL-properties-compact}, the family $\mathcal{W}:=\{ W_\alpha: \alpha\in \w^\w\}$ is a $\GG$-base for $G$. Moreover, for every $k\in\w$ and each $\alpha\in \w^\w$, (\ref{equ:IL-5}) implies
\[
U_{\alpha(n),n}\cdots U_{\alpha(1),1} \cdot U_{\alpha(0),0}^2 \cdot U_{\alpha(1),1}\cdots U_{\alpha(n),n} \subseteq D_k(\alpha).
\]
Hence $ \bigcup_{k\in\w} D_k(\alpha) = W_\alpha$ and therefore $\mathcal{W}$ satisfies $\dd$.


\smallskip
(ii) For every $n\in\w$, $G_n$ being metrizable is an $\aleph$-space. Fix a $\sigma$-locally finite $k$-network $\Nn_n =\bigcup_{k\in\w} \Nn_{k,n}$ for $G_n$. Set $\Nn:= \bigcup_{k,n\in\w} \Nn_{k,n}$.

We claim that $\Nn_{k,n}$ is locally finite in $G$ for every $k,n\in\w$. Indeed, fix arbitrarily $g\in G$. If $g\not\in G_n$, then the closeness of $G_n$ in $G$ (Theorem \ref{t:IL-properties-compact}) implies that there is a neighborhood $U$ of $g$ such that $U\cap N=\emptyset$ for every $N\in \Nn_{k,n}$. Assume that $g\in G_n$. Take a $\tau_n$-open neighborhood $V$ of $g$ such that the family $\{ N\in \Nn_{k,n}: V\cap N\not=\emptyset\}$ is finite. Since $G_n$ is a subgroup of $G$, choose a neighborhood $U$ of $g$ in $G$ such that $U\cap G_n\subseteq V$. It is clear that the family
\[
\{  N\in \Nn_{k,n}: U\cap N\not=\emptyset\} = \{  N\in \Nn_{k,n}: (U\cap G_n)\cap N\not=\emptyset\}
\]
is finite. Therefore $\Nn_{k,n}$ is locally finite.

The claim implies that $\Nn$ is $\sigma$-locally finite. Therefore it remains to prove that $\Nn$ is a $k$-network. Fix a compact subset $K$ of $G$ and an open neighborhood $U$ of $K$. By Theorem \ref{t:IL-properties-compact}, there is $m\in\w$ such that $K\subseteq G_m$. As $U\cap G_m$ is a $\tau_m$-open neighborhood of $K$, there is a finite subfamily $\FF$ of $\Nn_m$ such that $K\subseteq \bigcup \FF \subseteq U\cap K \subseteq U$. Thus $\Nn$ is a $k$-network. 
\end{proof}

Following Michael \cite{Mich}, a topological space $X$ is called an {\em $\aleph_0$-space} if  $X$ is regular and has a countable $k$-network.
\begin{corollary} \label{c:IL-P0-space}
Let  a closed tower $\{ G_n\}_{n\in\w}$ of metrizable groups satisfy $(GC)$.  Then the following assertions are equivalent:
\begin{enumerate}
\item[{\rm (i)}] $\glim G_n$ is Lindel\"{o}f;
\item[{\rm (ii)}] $\glim G_n$ is an $\aleph_0$-space;
\item[{\rm (iii)}] $\glim G_n$ is separable.
\end{enumerate}
\end{corollary}

\begin{proof}
Set $G:=\gglim G_n$.

(i)$\Rightarrow$(iii) For every $n\in\w$, the closed subspace $G_n$ of $G$ (Theorem \ref{t:IL-properties-compact}) is also Lindel\"{o}f. Since $G_n$ is metrizable it is separable (\cite[Theorem~4.1.16]{Eng}). If $D_n$ is a dense subset of $G_n$, it is clear that $\bigcup_{n\in\w} D_n$  is a dense subset of $G$. Thus $G$ is separable.

(iii)$\Rightarrow$(ii) Theorem \ref{t:IS-Pytkeev-G-base}(i) and Theorem 1.7 of \cite{GK-GMS2} imply that the family $\mathcal{D}:=\{ D_k(\alpha): k\in\w, \alpha\in \w^\w\}$ is a countable $cp$-network at the identity $e\in G$ (i.e., $\mathcal{D}$ satisfies the following property: for any subset $A\subseteq G$ with $e\in \overline{A}\setminus A$ and each neighborhood $\mathcal{O}$ of $e$ there is a set $D\in \mathcal{D}$ such that $e\in D \subseteq \mathcal{O}$ and $D\cap A$ is infinite). Thus $G$ is an $\aleph_0$-space by Theorem 1.7 of \cite{GK-GMS2} (which states that $G$ is even a $\Pp_0$-space).

Finally,  (ii) implies (i) since every $\aleph_0$-space is Lindel\"{o}f, see Proposition 3.12 of \cite{GK-GMS1}. 
\end{proof}

To detect when $\gglim G_n$ is an Ascoli space is a much more difficult problem. We shall use the following two assertions.
\begin{proposition}[\cite{Gab-LF}] \label{p:MKw-sequential}
If a topological group $(G,\tau)$ is an $\mathcal{MK}_\omega$-space, then it is either a locally compact metrizable group or is a sequential non-Fr\'{e}chet--Urysohn space.
\end{proposition}

\begin{proposition}[\protect{\cite[Theorem~3]{Yamasaki}}] \label{p:IL-open}
Let $\{ (G_n,\tau_n)\}_{n\in \w}$ be a tower of topological groups such that, for some $m\in\w$,  $G_m$ is open in $G_k$  for all $k>m$. Then $\tau_{ind}$ is a group topology on $G:= \bigcup_{n\in\w} G_n$ and $G_m$ is an open subgroup of $(G,\tau_{ind})$.
\end{proposition}


We need the following proposition. 
\begin{proposition} \label{p:inductive-LC-group}
Let $\{ (G_n,\tau_n)\}_{n\in\w}$ be a tower of locally compact groups. Then:
\begin{enumerate}
\item[{\rm (i)}] $\tau_{ind}=\tau_{BS}$ and they are group topologies;
\item[{\rm (ii)}]  the group $G:=(G,\tau_{ind})=\glim G_n$ contains an open $k_\w$-subgroup;
\item[{\rm (iii)}]   if additionally all $G_n$ are metrizable, then one of the following assertions holds:
    \begin{enumerate}
\item[$(iii)_1$] there is an $m\in\w$ such that $G_n$ is open in $G_{n+1}$ for every $n\geq m$, then $G$ is metrizable and locally compact;
\item[$(iii)_2$] $G$ contains an open $\mathcal{MK}_\omega$-subgroup and is a sequential non-Fr\'{e}chet--Urysohn space.
\end{enumerate}
\end{enumerate}
\end{proposition}

\begin{proof}
(i) is Theorem 2.7 of \cite{TSH}.

\smallskip
(ii) For every $n\in\w$, let $W_n$ be an open symmetric neighborhood of $e\in G_n$ with compact closure. Denote by $H_n$ the open $\sigma$-compact subgroup of $G_n$ generated by $W_n$. We can assume that $W_n\subseteq W_{n+1}$ and hence $H_n \subseteq H_{n+1}$ for every $n\in\w$. Note that since locally compact groups are complete, $G_n$ is a closed subgroup of $G_{n+1}$ for every $n\in\w$. Now (i) implies  $H:=(H,\tau^H_{BS})=\gglim H_n$, where $H:=\bigcup_{n\in\w} H_n$ and $\tau_{BS}^H$ is the Bamboo-Shoot topology on $H$. Moreover, the construction of the topologies $\tau^H_{BS}$ and $\tau_{BS}$ shows that $H$ is an open subgroup of $G$. Since all $H_n$ are  $k_\w$-groups, the definition of $\tau_{ind}^H$ implies that $(H,\tau_{ind}^H)$ is also a $k_\w$-space. Finally,  (i) implies $\tau_{ind}^H=\tau^H_{BS}$, and hence $H$ is an open $k_\w$-subgroup of $G$.

\smallskip
(iii) If there is an $m\in\w$ such that $G_n$ is open in $G_{n+1}$ for every $n\geq m$, then $G_m$ is an open subgroup of $G$ by Proposition \ref{p:IL-open}. Thus $G$ is metrizable and locally compact.

Assume that for infinitely many $n\in\w$, the group $G_n$ is not open in $G_{n+1}$. Since all the groups $G_n$ are metrizable, the group $H$ defined in the proof of (ii) is an $\mathcal{MK}_\omega$-space. By Proposition \ref{p:MKw-sequential}, $H$ is either locally compact or is a sequential non-Fr\'{e}chet--Urysohn space. However, $H$ cannot be locally compact since, otherwise, it would have a compact neighborhood $U$ of $e$. But then, by Theorem \ref{t:IL-properties-compact}, $U\subseteq H_k$ for some $k\in\w$. Therefore $H_k$ is an open subgroup of $H_n$ and hence of $G_n$ for all $n\geq k$ which contradicts our assumption on $G_n$. Since  $H$ is an open subgroup of $G$, we obtain that $G$ is also a sequential non-Fr\'{e}chet--Urysohn space. 
\end{proof}

We shall use also the following proposition to show that a space is not Ascoli.
\begin{proposition}[\protect{\cite[Proposition~2.1]{GKP}}] \label{p:Ascoli-sufficient}
Let $X$ be a Tychonoff space. Assume   $X$ admits a  family $\U =\{ U_i : i\in I\}$ of open subsets of $X$, a subset $A=\{ a_i : i\in I\} \subseteq X$ and a point $z\in X$ such that
\begin{enumerate}
\item[{\rm (i)}] $a_i\in U_i$ for every $i\in I$;
\item[{\rm (ii)}] $\big|\{ i\in I: C\cap U_i\not=\emptyset \}\big| <\infty$  for each compact subset $C$ of $X$;
\item[{\rm (iii)}] $z$ is a cluster point of $A$.
\end{enumerate}
Then $X$ is not an Ascoli space.
\end{proposition}

As we mentioned in the introduction, Yamasaki \cite{Yamasaki} showed that if a closed tower $\{ G_n\}_{n\in\w}$ of metrizable groups is such that $\tau_{ind}$ is a group topology, then there is $m\in\w$  such that, for all $n>m$, either $G_m$ is open in $G_n$  or $G_n$ is locally compact. Below we include a partial result of Yamasaki's theorem because it easily follows from the proof of the proposition.
\begin{proposition} \label{p:IL-Ascoli-inductive}
Let  a closed tower $\{ (G_n,\tau_n)\}_{n\in\w}$ of metrizable groups satisfy $(GC)$.
If $(G,\tau_{BS})$ is an Ascoli space or $\tau_{ind}$ is a group topology, then all the groups $G_n$ are locally compact.
\end{proposition}

\begin{proof}
Suppose for a contradiction that there is $G_i$, say $G_0$, which is not locally compact. Denote by $e$ the identity of $G$ (and hence all the $G_n$'s). For every $i\in\w$ we denote by $\rho_i$ a left invariant metric on $G_i$ and set
\[
B_{n,i} := \{ x\in G_i: \rho_i (x,e)<2^{-n} \}, \quad n\in\w.
\]

{\em Step 1.} Consider the open base of neighborhoods $\{ B_{n,0} \}_{n\in\w}$ of the unit $e$ of $G_0$, so $\overline{B_{n+1,0}}\subseteq B_{n,0}$. Then there is a strictly increasing sequence $\{n_k\}_{k\in \w}$ such that $n_{k+1}>n_k+1$ and for every $k\in\w$, the set $\overline{B_{n_k,0}}\setminus B_{n_k+1,0}$ is not compact. Indeed, otherwise, there would exist an $n_0\in\w$ such that $\overline{B_{n_0,0}}\setminus B_{n,0}$ is compact for all $n> n_0$. Since $B_{n,0}$ converges to $e$ (i.e., each neighborhood of $e$ contains all but finitely many of $B_{n,0}$'s), we obtain that $\overline{B_{n_0,0}}$ is compact. So $G_0$ is locally compact, a contradiction.

Set $P_k:=\overline{B_{n_k,0}}\setminus B_{n_k+1,0}$. Then $P_k$ is metrizable and non-compact, and hence $P_k$ is not pseudocompact (see \cite[4.1.17]{Eng}). By \cite[Theorem~3.10.22]{Eng},  there exists a locally finite collection $\{W_{n,k}\}_{n\in \w}$ of nonempty open subsets  of $P_k$. Since $G_0$ being non-locally compact is not discrete, any $W_{n,k}$ contains a point from $B_{n_k,0}\setminus \overline{B_{n_k+1,0}}$. Therefore we can assume in addition that  $\overline{W_{n,k}}\subseteq \mathrm{Int}(P_k)$ for every $k\in\w$. As $n_{k+1}>n_k+1$, it is easy to see that, for every $m\in\w$,  the family
\[
\mathcal{W}_m := \{ W_{n,i}: n\in\w, i\leq m\}
\]
is also locally finite in $X_0$. For every $n,k\in\w$, pick arbitrarily a point $x_{n,k}\in W_{n,k}$.

\medskip
{\em Step 2.} We claim that for every $k\geq 1$ there are
\begin{enumerate}
\item[{\rm (a)}] a one-to-one sequence $\{ y_{n,k}\}_{n\in\w}$ in $G_k\setminus G_{k-1}$ converging to the unit $e\in G$;
\item[{\rm (b)}] for every $n\in\w$, an open neighborhood $U_{n,k}$ of $a_{n,k} :=x_{n,k-1}\cdot  y_{n,k}$ in $G$;
\end{enumerate}
such that
\begin{enumerate}
\item[{\rm (c)}] $U_{n,k}\cap G_{k-1} =\emptyset$ for every $n\in\w$;
\item[{\rm (d)}] the family
\[
\mathcal{V}_k := \{ U_{n,i} \cap X_k: \; n\in\w, \, 1\leq i\leq k\}
\]
is locally finite in $G_k$.
\end{enumerate}

Indeed, for every $k\geq 1$, let $\{ y_{n,k}\}_{n\in\w}$ be an arbitrary one-to-one sequence in $G_k\setminus G_{k-1}$  converging to $e$ (such a sequence exists because $G_{k-1}$ is not open in $G_k$). For every $k\geq 1$ and each $n\in\w$,   choose an open symmetric neighborhood $V_{n,k}$ of $e$ in $G$ such that (recall that $G_{k-1}$ is a closed subgroup of $G$ by (ii) of Theorem \ref{t:IL-properties-compact})
\begin{enumerate}
\item[$(\alpha)$] $V_{n,k}^3 \cap G_i \subseteq B_{n,i}$ for every $0\leq i\leq n$;
\item[$(\beta)$] $\big( y_{n,k} \cdot V_{n,k}^3\big) \cap G_{k-1} =\emptyset$.
\end{enumerate}
For every $k\geq 1$ and each $n\in\w$,   set
\[
a_{n,k}:= x_{n,k-1}\cdot  y_{n,k} \; \mbox{ and }\; U_{n,k} :=a_{n,k} V_{n,k}.
\]
Clearly, (a) and (b) are fulfilled.  Also (c) is  fulfilled since if $U_{n,k}\cap G_{k-1} \not=\emptyset$ for some $k\geq 1$  and $n\in\w$, then $x_{n,k-1} y_{n,k} v \in G_{k-1}$ for some $v\in V_{n,k}$. So $y_{n,k} v \in x_{n,k-1}^{-1} G_{k-1} =G_{k-1}$ that contradicts $(\beta)$.

\medskip

Let us check (d). Fix $k\geq 1$ and $x\in G_k$. Consider the two possible cases.

\medskip
{\em Case 2.1. Let $x\in G_k\setminus G_{0}$.} Then $\rho_k(x, G_{0})>0$ since $G_0$ is closed in $G_k$. For every $1\leq i\leq k$, since $y_{n,i} \to e$ in $G$ and $x_{n,i-1}\in G_0$, the condition $(\alpha)$ implies $V_{n,i} \cap G_k \subseteq B_{n,k}$ for every $n >k$. Therefore, for every $1\leq i\leq k$, we obtain
\[
\begin{split}
\lim_n \rho_k \big( U_{n,i} \cap G_k, G_0\big) & =\lim_n \rho_k \big( y_{n,i}V_{n,i} \cap G_k, G_0\big) \quad (\mbox{since } y_{n,i}\in G_k) \\
& =\lim_n \rho_k \big( y_{n,i}(V_{n,i} \cap G_k), G_0\big) \leq \lim_n \rho_k \big( y_{n,i}B_{n,k}, e\big)= 0.
\end{split}
\]
Hence there is an open neighborhood $U_x$ of $x$ in $G$ such that $U_x\cap G_k$ intersects only a finite subfamily of $\mathcal{V}_k$.

\medskip
{\em Case 2.2. Let $x\in G_0$.} Choose an open symmetric neighborhood $U_x$ of $e$ in $G$ such that $xU_x^3 \cap G_0$ intersects only with a finite subfamily of $\mathcal{W}_k$. We claim that $xU_x \cap G_k$ intersects only a finite subfamily of $\mathcal{V}_k$. Indeed, assuming the converse we can find $1\leq i\leq k$ such that
\[
(xU_x \cap G_k) \cap (U_{n,i}\cap G_k)\not=\emptyset
\]
for every $n\in I$, where $I$ is an infinite subset of $\w$. Then for every $n\in I$ there are $u_n\in G_k$, $t_n\in U_x$ and $z_n\in V_{n,i}\cap G_k$ such that
\[
u_n =x\cdot t_n = x_{n,i-1} y_{n,i} z_n.
\]
Note that, by $(\alpha)$, $z_n= y_{n,i}^{-1} x_{n,i-1}^{-1} u_n \in V_{n,i}\cap G_k$ belongs to $U_x \cap G_k$ for all sufficiently large $n\in I$, and also $y_{n,i}\in U_x \cap G_k$ for all sufficiently large $n\in I$ because $y_{n,i}\to e$. So
\[
x_{n,i-1} = x\cdot \left( t_n  z_n^{-1} y_{n,i}^{-1}\right) \in \big( x U_x^3 \cap G_0\big) \cap W_{n,i-1}
\]
for all sufficiently large $n\in I$. But this contradicts the choice of $U_x$.

Cases 2.1 and 2.2 show that $\mathcal{V}_{k}$  is locally finite in $G_{k}$.

\medskip
{\em Step 3. Assume that $(G,\tau_{BS})$ is an Ascoli space.} To get a contradiction with the assumption that $G_0$ is not locally compact we show that the families
\[
A:=\{ a_{i,k}: i\in\w, k\geq 1\}, \quad \U :=\{ U_{i,k}: i\in\w, k\geq 1\}
\]
and $z:=e$ satisfy (i)-(iii) of Proposition \ref{p:Ascoli-sufficient}.

\smallskip
Indeed, (i) is clear. To check (ii) let $K$ be a compact subset of $G$. By Theorem \ref{t:IL-properties-compact}, there is an $m\in\w$ such that $K\subseteq G_m$. So (c) implies that if $K\cap U_{n,i}\not=\emptyset$, then  $i\leq m$, and hence $U_{n,i}\cap G_m \in \mathcal{V}_m$. Since the family $\mathcal{V}_m$ is locally finite in $G_m$, we obtain that $K$ intersects only a finite subfamily of $\U$ that proves (ii).

To prove (iii) let $V$ be an open neighborhood of $e$ in $G$. Take an open neighborhood $U$ of $e$ such that $U^2\subseteq V$, and choose $k_0\in\w$ such that $W_{i,k_0}\subseteq G_0 \cap U$ for every $i\in\w$.  So $x_{i,k_0}\in U$ for every $i\in\w$. Since $\lim_i y_{i,k_0+1} =e$ we obtain that $a_{i,k_0+1}=x_{i,k_0} y_{i,k_0+1} \in U\cdot U\subseteq V$ for all sufficiently large $i$. Thus $e\in \overline{A}$ and (iii) holds. Finally, Proposition \ref{p:Ascoli-sufficient} implies that the group $X$ is not Ascoli which is a desired contradiction.

\medskip
{\em Step 4. Assume that $\tau_{ind}=\tau_{BS}$.} To get a contradiction with the assumption that $G_0$ is not locally compact, it is sufficient to show that  $e\not\in \overline{A}^{\,\tau_{ind}}$. Indeed, by Step 3, we know that $e\in \overline{A}^{\,\tau_{BS}}$, and then (ii) of Lemma \ref{l:BS-description} shows that  $\tau_{ind}$ is not a group topology.

To prove that $e\not\in \overline{A}^{\,\tau_{ind}}$, for every $k\geq 1$, set $A_k :=\{ a_{i,k}=x_{i,k-1}\cdot  y_{i,k}:  i\in\w\}$. Since the sequence $\{ x_{i,k-1}: i\in\w\}$ is closed in $G_0$ and $\lim_i y_{i,k} =e$ we obtain that $A_k$ is closed in $G_k$. Now (a) implies that
\[
A\cap G_k =(A_1 \cup\cdots \cup A_k) \cap G_k
\]
is closed in $G_k$ for every $k\geq 1$. Hence $A$ is closed in $\tau_{ind}$. Thus $e\not\in \overline{A}^{\,\tau_{ind}}$. 
\end{proof}

\begin{theorem} \label{t:Ascoli-inductive-1}
Let  a closed tower  $\{ G_n\}_{n\in\w}$ of metrizable groups satisfy $(GC)$. Set $G:=\glim G_n$. Then the following assertions are equivalent:
\begin{enumerate}
\item[{\rm (i)}] $G$ is an Ascoli space;
\item[{\rm (ii)}] $\tau_{ind}$ is a group topology on $G$;
\item[{\rm (iii)}] one of the following assertions holds:
\begin{enumerate}
\item[$(iii)_1$] there is an $m\in\w$ such that $G_n$ is open in $G_{n+1}$ for every $n\geq m$, so $G$ is metrizable;
\item[$(iii)_2$] all the groups $G_n$ are locally compact, so $G$ contains an open $\mathcal{MK}_\omega$-subgroup and is a sequential non-Fr\'{e}chet--Urysohn space.
\end{enumerate}
\end{enumerate}
\end{theorem}

\begin{proof}
(i)$\Rightarrow$(iii) and (ii)$\Rightarrow$(iii): If there is an $m\in\w$ such that $G_n$ is open in $G_{n+1}$ for every $n\geq m$, then $G_m$ is an open subgroup of $G$ by Proposition \ref{p:IL-open}. Thus $G$ is metrizable.

Assume that for infinitely many $n\in\w$ the group $G_n$ is not open in $G_{n+1}$. Without loss of generality we can assume that $G_n$ is not open in $G_{n+1}$ for all $n\in\w$. Then Proposition \ref{p:IL-Ascoli-inductive} imply that all the groups $G_n$ are locally compact. Our assumption on $\{ G_n\}$ and Proposition \ref{p:inductive-LC-group} imply that $G$ contains an open $\mathcal{MK}_\omega$-subgroup and is a sequential non-Fr\'{e}chet--Urysohn space.

(iii)$\Rightarrow$(i) follows from the Ascoli theorem \cite[Theorem~3.4.20]{Eng}, and (iii)$\Rightarrow$(ii) follows from Propositions \ref{p:IL-open} and \ref{p:inductive-LC-group}. 
\end{proof}



\bibliographystyle{amsplain}

\end{document}